\documentclass[10pt]{amsart}

\usepackage{amsfonts,amssymb,amscd,amstext}
\usepackage[utf8]{inputenc}
\usepackage[paper size={210mm,297mm},left=40mm,right=40mm,top=50mm,bottom=50mm]{geometry}
\usepackage[charter]{mathdesign}
\usepackage{graphicx}
\usepackage[colorlinks=true,linkcolor=blue,citecolor=red]{hyperref}

\renewcommand{\le}{\leqslant}
\renewcommand{\ge}{\geqslant}
\renewcommand{\le}{\leqslant}
\renewcommand{\ge}{\geqslant}
\newcommand{\ptl}{\partial}
\newcommand{\rr}{{\mathbb{R}}}
\newcommand{\la}{\lambda}
\newcommand{\hh}{H}
\newcommand{\h}{H}
\newcommand{\esf}{\mathbb{S}}

\newcommand{\nn}{\mathbb{N}}
\newcommand{\pp}{P}

\newcommand{\Om}{\Omega}
\newcommand{\eps}{\varepsilon}

\newcommand{\vol}[1]{|#1|}

\newcommand{\cl}[1]{\text{\rm cl}(#1)}
\newcommand{\clb}{\mkern2mu\overline{\mkern-2mu B\mkern-2mu}\mkern2mu}
\newcommand{\cld}{\mkern2mu\overline{\mkern-2mu D\mkern-2mu}\mkern2mu}

\DeclareMathOperator{\divv}{div}

\DeclareMathOperator{\intt}{int}
\DeclareMathOperator{\Lip}{Lip}

\DeclareMathOperator{\inr}{inr}

\DeclareMathOperator{\diam}{diam}

\newtheorem{theorem}{Theorem}[section]
\newtheorem{proposition}[theorem]{Proposition}
\newtheorem{lemma}[theorem]{Lemma}
\newtheorem{corollary}[theorem]{Corollary}

\theoremstyle{definition}

\newtheorem{remark}[theorem]{Remark}

\theoremstyle{remark}

\newenvironment{enum}{\begin{enumerate}
}{\end{enumerate}}

\numberwithin{equation}{section}

\begin{document}

\title[Isoperimetric inequalities in convex cylinders and cylindrically bounded convex bodies]{Isoperimetric inequalities in convex cylinders and cylindrically  bounded convex bodies}

\author[M.~Ritor\'e]{Manuel Ritor\'e} \address{Departamento de
Geometr\'{\i}a y Topolog\'{\i}a \\
Universidad de Granada \\ E--18071 Granada \\ Espa\~na}
\email{ritore@ugr.es}
\author[S.~Vernadakis]{Efstratios Vernadakis} \address{Departamento de
Geometr\'{\i}a y Topolog\'{\i}a \\
Universidad de Granada \\ E--18071 Granada \\ Espa\~na}
\email{stratos@ugr.es}

\date{\today}

\thanks{Both authors have been supported by MICINN-FEDER grant MTM2010-21206-C02-01, and Junta de Andaluc\'{\i}a grants FQM-325 and P09-FQM-5088}

\begin{abstract}
In this paper we consider the isoperimetric profile of convex cylinders $K\times\rr^q$, where $K$ is an $m$-dimensional convex body, and of cylindrically bounded convex sets, i.e, those with a relatively compact orthogonal projection over some hyperplane of $\rr^{n+1}$, asymptotic to a right convex cylinder of the form $K\times\rr$, with $K\subset\rr^n$. Results concerning the concavity of the isoperimetric profile, existence of isoperimetric regions, and geometric descriptions of isoperimetric regions for small and large volumes are obtained.
\end{abstract}

\subjclass[2000]{49Q10, 49Q20, 52B60}
\keywords{Isoperimetric inequalities, isoperimetric profile, unbounded convex bodies, convex cylinders}

\maketitle

\thispagestyle{empty}

\bibliographystyle{abbrv}


\section{Introduction}
In these notes we consider the \emph{isoperimetric problem} of minimizing perimeter under a given volume constraint inside a \emph{cylindrically bounded convex body}, an unbounded closed convex set $C\subset\rr^{n+1}$ with interior points and relatively compact projection onto the hyperplane $x_{n+1}=0$. The perimeter considered here will be the one relative to the interior of $C$. A way to deal with this isoperimetric problem is to consider the \emph{isoperimetric profile} of $C$, i.e., the function assigning to each $v>0$ the infimum of the perimeter of the sets inside $C$ of volume $v$. If this infimum is achieved for some set, this will be called an \emph{isoperimetric region}. The isoperimetric profile can be understood as an optimal isoperimetric inequality on $C$.


A cylindrically bounded convex set is always included and asymptotic, in a sense to be precised later, to a \emph{convex right cylinder}, a set of the form $K\times \rr$, where $K\subset \rr^n$ is a convex body. Here we have identified $\rr^n$ with the hyperplane $x_{n+1}=0$ of $\rr^{n+1}$. In this work we first consider the more general convex cylinders of the form $C=K\times\rr^q$, where $K\subset\rr^m$ is an arbitrary convex body with interior points, and $\rr^m\times\rr^q=\rr^{n+1}$, and prove a number of results for their isoperimetric profiles. No assumption on the regularity of $\ptl C$ will be made. Existence of isoperimetric regions is obtained in Proposition~\ref{prp:exicyl} following the scheme of proof by Galli and Ritoré \cite{gr}, which essentially needs a uniform local relative isoperimetric inequality \cite{rv1}, a doubling property on $K\times \rr^q$ given in Lemma~\ref{lem:doubling}, an upper bound for the isoperimetric profile of $C$ given in \eqref{prp:ICleICmin}, and a well-known deformation controlling the perimeter in terms of the volume. A proof of existence of isoperimetric regions in Riemannian manifolds with compact quotient under their isometry groups was previously given by Morgan \cite{morgmt}. Regularity results in the interior follow from Gonzalez, Massari and Tamanini \cite{MR684753} and Morgan \cite{MR1997594}, but no boundary regularity result is known for general convex bodies. We also prove in Proposition~\ref{prp:Concavdeperfcyldro} that the isoperimetric profile $I$ of a convex cylinder, as well as its power $I^{(n+1)/n}$, are concave functions of the volume, a strong result that implies the connectedness of isoperimetric regions. Further assuming $C^{2,\alpha}$ regularity of the boundary of $C$, we prove in Theorem~\ref{thm:sz} that, for an isoperimetric region $E\subset C$, either the closure of $\ptl E\cap\intt(C)$ is connected, or $E\subset K\times\rr$ is a slab. This follows from the connectedness of isoperimetric regions and from the results by Stredulinsky and Ziemer \cite{st-zi}. Next we consider small and large volumes. 
For small volumes, following Ritoré and Vernadakis \cite{rv1}, we show in Theorem~\ref{thm:optinsmalvol} that the isoperimetric profile of a convex cylinder for small volumes is asymptotic to the one of its narrowest tangent cone. As a consequence, we completely characterize the isoperimetric regions of small volumes in a convex prism, i.e, a cylinder $P\times \rr^q$ based on a convex polytope $P\subset\rr^m$. Indeed, we show in Theorem~\ref{thm:polytops} that the only isoperimetric regions of sufficiently small volume inside a convex prism are geodesic balls centered at the vertices with  tangent cone of the smallest possible solid angle. For large volumes, we shall assume that $C$ is a \emph{right} convex cylinder, i.e., $p=1$. Adapting an argument by Duzaar and Stephen \cite{du-st} to the case when $\ptl K$ is not smooth, we prove in Theorem~\ref{thm:isocyl} that for large volumes the only isoperimetric regions in $K\times\rr$ are the slabs $K\times I$, where $I\subset\rr$ is a compact interval. The case $K\times\rr^q$, with $q>1$, is  more involved and will be treated in a different paper (see \cite{rv2} for a proof for the Riemannian product $M\times\rr^k$, where $M$ is a compact Riemannian manifold without boundary). 

In the second part of this paper we apply the previous results for right convex cylinders to obtain properties of the isoperimetric profile of cylindrically bounded convex bodies. In Theorem~\ref{thm:asym-profile} we show that the isoperimetric profile of a cylindrically bounded convex body $C$ approaches, when the volume grows, that of its asymptotic half-cylinder. We also show the continuity of the isoperimetric profile in Proposition~\ref{prp:coniclycontis}. Further assuming $C^{2,\alpha}$ regularity of both the cylindrically bounded convex body $C$ and of its asymptotic cylinder, we prove the concavity of $I_C^{(n+1)/n}$ and existence of isoperimetric regions of large volume in Proposition~\ref{prp:excylindical}. Our final result, Theorem~\ref{thm:thmcyli}, implies that translations of isoperimetric regions of  unbounded volume converge in Hausdorff distance to a half-slab in the asymptotic half-cylinder. The same convergence result holds for their free boundaries, that converge in Hausdorff distance to a flat $K\times\{t\}$, $t\in  \rr^+$. Theorem~\ref{thm:thmcyli} is obtained from a clearing-out result for isoperimetric regions of large volume proven in Theorem~\ref{thm:leon rigot lem 42cyl} and its main consequence, lower density estimates for isoperimetric regions of large volume given in Proposition~\ref{prp:lodenbn}. Such lower density bounds provide an alternative proof of Theorem~\ref{thm:isocyl}, given in  Corollary~\ref{cor:isocyl}.

We have organized this paper into four sections. The next one contains basic preliminaries, while Sections~\ref{sec:cylinders} and \ref{sec:cylindrically} cover the already mentioned results for cylinders and cylindrically bounded sets, respectively.

\section{Preliminaries}
\label{sec:preliminaries}
A \emph{convex body} is a compact convex set with non-empty interior. If compact is replaced by closed and unbounded, we get an \emph{unbounded convex body}. We refer to Schneider's monograph \cite{sch} for background on convex sets.

The $s$-dimensional Hausdorff measure in $\rr^{n+1}$ will be denoted by $H^s$, for any $s\in\nn$. For $E \subset C$, the relative boundary of $E$ in the interior of $C$ is $\ptl_{C}E=\ptl E\cap \intt {C}$. The $(n+1)$-dimensional Hausdorff measure of $E$, $H^{n+1}(E)$ will be denoted by $\vol{E}$ and referred to as the \emph{volume} of $E$. Moreover, for every $x\in C$ and $r>0$ we shall define the intrinsic open ball $B_C(x,r)=B(x,r)\cap \intt {C}$, where $B(x,r)$ denotes the open Euclidean geodesic ball centered at $x$ of radius $r$. The closure of a set $E\subset\rr^{n+1}$ will be denoted by $\cl{E}$.

We also define the relative \emph{perimeter} of $ E$ in the interior of $C$ by
\[
\pp_C(E) = \sup \Big \{ \int_E\divv \xi\, d{\h}^{n+1}, \xi \in \Gamma_0(C) ,\, |\xi| \le 1 \Big \},
\]
where $\Gamma_0(C)$ is the set of smooth vector fields with compact support in $\intt {C}$. Observe that we are only computing the ${H}^n$-measure of $\ptl E$ inside the interior of $C$.  We shall say that $E$ has finite perimeter in the interior of $C$, or simply that $E\subset C$ has finite perimeter, if $\pp_C(E)<\infty$. We refer the reader to Maggi's monograph \cite{MR2976521} for background on finite perimeter sets.

If $C,C'\subset\rr^{n+1}$ are convex bodies (possible unbounded) and $f:C\to C'$ is a Lipschitz map, then, for every $s>0$ and $E\subset C$, we get $\h^s(f(E)) \le \Lip(f)^s\, \h^s(E)$. Furthermore, $f(\ptl_{C}E)=\ptl_{f(C)}(f(E))$. Thus we obtain
\begin{lemma}
\label{lem:bilip}
Let $C$, $C'\subset\rr^{n+1}$ be (possibly unbounded) convex bodies, and $f:C\to C'$  a bilipschitz map. Then we have
\begin{equation}
\label{eq:bilip}
\begin{split}
\Lip(f ^{-1})^{-n}\,\pp_C(E) \le \pp_{f(C)}(f(E)) \le \Lip(f)^n\, \pp_C(E)
\\
\Lip(f ^{-1})^{-(n+1)}\,\vol{E} \le \vol{f(E)} \le \Lip(f)^{n+1}\,\vol{E}.
\end{split}
\end{equation}
\end{lemma}
\begin{remark}
\label{rem:lipcomp}
 If $M_ i$, $i=1,2,3$ are metric spaces and $f_i: M_i\to M_{i+1}$, $i=1,2$ are lipschitz maps, then $\Lip(f_2\circ f_1)\le \Lip(f_1)\Lip(f_2)$. Consequently if $g: M_1\to M_2$ is a bilipschitz map, then $1\le \Lip(g)\Lip(g ^{-1})$.
\end{remark}

Given a (possibly unbounded) convex body, we define the \emph{isoperimetric profile} of $C$ by
\begin{equation}
\label{eq:profile}
I_C(v)=\inf \Big \{ \pp_C(E) : E \subset C, \vol{E} = v \Big \}.
\end{equation}
We shall say that $E \subset C$ is an \emph{isoperimetric region} if $\pp_C(E)=I_C(\vol{E})$. The \emph{renormalized isoperimetric profile} of $C$ is given by
\begin{equation}
\label{eq:renprofile}
I_C^{(n+1)/n}.
\end{equation}
Lower semicontinuity of perimeter and standard compactness results for finite perimeter sets imply that isoperimetric regions exist in a fixed bounded subset of Euclidean space.

The known results on the regularity of isoperimetric regions are summarized in the following Lemma. One can always assume that a representative of an isoperimetric region is chosen so that it is closed, includes its points of density one and does not contain the points of density zero.

\begin{lemma}[{\cite{MR684753}, \cite{MR862549}, \cite[Thm.~2.1]{MR1674097}}]
\label{lem:n-7}
\mbox{}
Let $C\subset \rr^{n+1}$ be a (possible unbounded) convex body and $E\subset C$ an isoperimetric region.
Then $\ptl_C E = S_0\cup S$, where  $S_0\cap S=\emptyset$ and
\begin{enum}
\item $S$ is an embedded $C^{\infty}$ hypersurface of constant mean curvature.\item $S_0$ is closed and $H^{s}(S_0)=0$ for any $s>n-7$.
\end{enum}
Moreover, if the boundary of $C$ is of class $C^{2,\alpha}$ then $\cl{\ptl E\cap\intt(C)}=S\cup S_0$, where
\begin{enum}
\item[(iii)] $S$ is an embedded $C^{2,\alpha}$ hypersurface of constant mean curvature.
\item[(iv)] $S_0$ is closed and $H^s(S_0)=0$ for any $s>n-7$.
\item[(v)] At points of $ S \cap \ptl C$,  $ S$ meets $\ptl C$ orthogonally.
\end{enum}
\end{lemma}


The concavity of $I_C$ and $I_C^{(n+1)/n)}$ for a convex body, \cite{MR2008339}, \cite[Cor.~6.11]{MR2507637}, \cite[Cor.~4.2]{rv1}, imply

\begin{lemma}[{\cite[Lemma~4.9]{rv1}}]
\label{lem:I_C(v)> cv}
Let $C\subset \rr^{n+1}$ be a convex body and $0<v_0<\vol{C}$. Then
\begin{equation}
\label{eq:I_C(v)> cv.}
I_C(v)\ge \frac{I_C(v_0)}{v_0}\,v\quad\text{and}\quad I_C(v)\ge \frac{I_C(v_0)}{v_0^{n/(n+1)}}\,v^{n/(n+1)},
\end{equation}
for all $0\le v\le v_0$.
\end{lemma}

We also have the following uniform relative isoperimetric inequality and bounds on the volume of relative balls in convex cylinders.

\begin{proposition}
\label{prp:heramientas para existencia}
\mbox{}
Let $C=K\times\rr^q$, where $K$ is an $m$-dimensional  convex body. Given $r_0>0$, there exist positive constants $M$, $\ell_1$, only depending on $r_0$ and $C$, and a universal positive constant $\ell_2$ so that
\begin{equation}
\label{eq:isnqgdbl1}
\pp_{\clb_C(x,r)}(v)\ge M\, {\min \{v,\vol{\clb_C(x,r)}-v\}}^{n/(n+1)},\end{equation}
for all $x\in C$, $0<r\le r_0$, and $0<v<\vol{\clb(x,r)}$, and
\begin{equation}
\label{eq:isnqgdbl1a}
\ell_1 r^{n+1} \le \vol{\clb_C(x,r)} \le \ell_2 r^{n+1},
\end{equation}
for any $x\in C$, $0<r\le r_0$.
\end{proposition}

\begin{proof}
Since the quotient of $C$ by its isometry group is compact, the proof  is reduced to that of \cite[Thm.~4.12]{rv1}.
\end{proof}

Let $K\subset \rr^{n+1}$ be a closed convex cone with vertex $p$ . Let $\alpha(K)=\hh^n(\ptl{B}(p,1)\cap \intt(K))$ be the \emph{solid angle} of $K$. It is known that intrinsic geodesic balls (Euclidean balls intersected with $K$) centered at the vertex are isoperimetric regions in $K$, \cite{lions-pacella}, \cite{r-r}, and that they are the only ones \cite{FI} for general convex cones, without any regularity assumption on the boundary. The isoperimetric profile of $K$ is given by
\begin{equation}
\label{eq:isopsolang}
I_K(v)={\alpha(K)}^{1/(n+1)}\,(n+1)^{n/(n+1)}v^{n/(n+1)}.
\end{equation}
Consequently the isoperimetric profile of a convex cone is completely determinated by its solid angle.

We define the tangent cone $C_{p}$ of a convex body $C$ at a given boundary point $p\in\ptl C$ as the closure of the set
\[
\bigcup_{\la > 0} h_{p,\la} (C),
\]
where $ h_{p,\la}$ is the dilation of center $p$ and factor $\la$. Since the quotient of the cylinder $C=K\times\rr^q$ by its isometry group is compact, then adapting \cite[Lemma 6.1]{rv1} we get  the existence of points in $\ptl C$ whose tangent cones are minima of the solid angle function. By \eqref{eq:isopsolang}, the isoperimetric profiles of tangent cones which are minima of the solid angle function coincide. The common profile will be denoted by $I_{C_{\min}}$.

\begin{proposition}[{\cite[Proposition 6.2]{rv1}}]
\label{prp:ICleICmin}
Let $C\subset\rr^{n+1}$ be a convex body (possibly unbounded), $p\in C$ and let  $H\subset\rr^{n+1}$ denote the closed half-space, then
\begin{equation}
\label{eq:ICleICmin}
I_C(v)\le  I_{C_p}(v)\le  I_H(v),
\end{equation}
for all $0\le v\le\vol{C}$. Moreover $I_C\le I_{C_{\min}}$.
\end{proposition}

\begin{remark}
Proposition~\ref{prp:ICleICmin} asserts that $\overline{E}\cap\ptl C\neq\emptyset$ when $E\subset C$ is an isoperimetric region since, in case $\overline{E}\cap\ptl C$ is empty, then $\overline{E}$ is an Euclidean ball.
\end{remark}

Let $C\subset\rr^{n+1}$ be a closed unbounded convex set. Assume there exists a hyperplane $\Pi$ such that the orthogonal projection $\pi:\rr^{n+1}\to\Pi$ takes $C$ to a bounded set, and let $K$ be the bounded convex body defined as the closure of $\pi(C)$. Then $C$ is contained in the right convex cylinder $\text{cyl}(K)$ of base $K$. Since $C$ is unbounded, it contains a half-line which is necessarily parallel to the axis of $\text{cyl}(K)$. If $C$ contains a complete line, then $C=\text{cyl}(K)$ is a cylinder \cite[Lemma~1.4.2]{sch}. This implies that $C$ is either a cylinder, or is contained in a half-cylinder. We shall say that the unbounded convex body $C$ is \emph{cylindrically bounded} if it is contained in a convex cylinder of bounded base, and it is not a cylinder itself.

\section{Isoperimetric regions in cylinders}
\label{sec:cylinders}
In this Section we consider the isoperimetric problem when the ambient space is a convex cylinder $K\times \rr^q$, where $K\subset\rr^m$ is a convex body. We shall assume that $m+q=n+1$. Existence of isoperimetric regions in $K\times\rr^q$ can be obtained following the strategy of Galli and Ritoré for contact sub-Riemannian manifolds \cite{gr} with compact quotient under their contact isometry group. One of the basic ingredients in this strategy is the relative isoperimetric inequality in Proposition~\ref{prp:heramientas para existencia}. A second one is the property that any unbounded convex body $C$ is a doubling metric space

\begin{lemma}
\label{lem:doubling}
Let $C\subset \rr^{n+1}$ be an unbounded convex body. Then
\begin{equation}
\label{eq:doubling}
|B_C(x,2r)|\le (2^{n+1}+1) |B_C(x,r)|,
\end{equation}
for any $x\in C$ and any $r>0$.
\end{lemma}

\begin{proof}
Let $x\in C$, $r>0$ and let $K$ denote the closed cone with vertex $x$ subtended by the closure of $\ptl B_C(x,r)$. Then
\begin{equation*}
\begin{split}
\vol{B_C(x,2r)}& = \vol{B_C(x,2r)\setminus B_C(x,r)}+\vol{B_C(x,r)}
\\
&\le  \vol{B_K(x,2r)\setminus B_K(x,r)}+\vol{B_C(x,r)}
\\
&\le  \vol{B_K(x,2r)}+\vol{B_C(x,r)}
\\
&= 2^{n+1} \vol{B_K(x,r)}+\vol{B_C(x,r)}
\\
&\le (2^{n+1}+1) \vol{B_C(x,r)},
\end{split}
\end{equation*}
as we claimed. Inequality $\vol{B_C(x,2r)\setminus B_C(x,r)}\le \vol{B_K(x,2r)\setminus B_K(x,r)}$ follows since, by convexity and the definition of $K$, $B_C(x,2r)\setminus B_C(x,r)\subset B_K(x,2r)\setminus B_K(x,r)$.
\end{proof}

%

Using Lemma~\ref{lem:doubling} and Proposition~\ref{prp:ICleICmin} we can show

\begin{proposition}
\label{prp:exicyl}
Consider the convex cylinder $C=K\times \rr^q$, where $K\subset \rr^m$ is a convex body. Then isoperimetric regions exist in $K\times\rr^q$ for all volumes and they are bounded.
\end{proposition}

\begin{proof}
To follow the strategy of Galli and Ritoré \cite{gr} (see Morgan \cite{morgmt} for a slightly different proof for smooth Riemannian manifolds), we only need a relative isoperimetric inequality \eqref{eq:isnqgdbl1} for balls $\overline{B}_C(x,r)$ of small radius with a uniform constant; the doubling property \eqref{eq:doubling}; inequality \eqref{eq:ICleICmin} giving an upper bound of the isoperimetric profile; and a deformation of isoperimetric sets $E$ by finite perimeter sets $E_t$ satisfying
\begin{equation*}
|\hh^n(\ptl E_t\cap\intt(C))-\hh^n(\ptl E\cap\intt(C))|\le M\,|\vol{E_t}-\vol{E}|,
\end{equation*}
for small $|t|$ and some constant $M>0$ not depending in $t$, which can be obtained by deforming the regular part of the boundary of $E$ using the flow associated to a vector field with compact support.

Using all these ingredients, the proof of Theorem~6.1 in \cite{gr} applies to prove existence of isoperimetric regions in $K\times\rr^q$.
\end{proof}

Let us prove now the concavity of the isoperimetric profile of the cylinder and of its power $\tfrac{n+1}{n}$. We start by proving its continuity.

\begin{proposition}
\label{prp:I_{Cyl} is contin}
Let $C=K\times \rr^q$, where $K$ is an $m$-dimensional convex body. Then $I_C$ is non-decreasing and continuous.
\end{proposition}

\begin{proof}
Given $t>0$, the smooth map $\varphi_t:C\to C$ defined by $\varphi_t(x,y)=(x,ty)$, $x\in K$, $y\in\rr^q$, satisfies $|\varphi_t(E)|=t^q\,|E|$. When $t\le 1$, we also have $P_C(\varphi_t(E))\le t^{q-1}\,P_C(E)$. This implies that the isoperimetric profile is a non-decreasing function: let $v_1<v_2$, and $E\subset C$ an isoperimetric region of volume $v_2$. Let $0<t<1$ so that $\vol{\varphi_t(E)}=v_1$. We have
\[
I_C(v_1)\le P(\varphi_t(E))\le P(E)=I_C(v_2).
\]
This shows that $I_C$ is non-decreasing.

Let us prove now the right-continuity of $I_C$ at $v$. Consider an isoperimetric region $E$ of volume $v$. We can take a smooth vector field $Z$ with support in the regular part of the boundary of $E$ such that $\int_E\divv Z\neq 0$. The flow $\{\varphi_t\}_{t\in\rr}$ of $Z$ satisfies $(d/dt)|_{t=0} |\varphi_t(E)|\neq 0$. Using the Inverse Function Theorem we obtain a smooth family $\{E_w\}$, for $w$ near $v$, with $|E_w|=w$ and $E_v=E$. The function $f(w)=P(E_w)$ satisfies $f\ge I_C$ and $I_C(v)=f(v)$. This implies that $I_C$ is right-continuous at $v$ since, for $v_i\downarrow v$, we have
\[
I_C(v)=f(v)=\lim_{i\to\infty} f(v_i)\ge \lim_{i\to\infty} I_C(v_i)\ge I_C(v),
\]
by the monotonicity of $I_C$.

To prove the left-continuity of $I_C$ at $v$ we take a sequence of isoperimetric regions $E_i$ with $v_i=|E_i|\uparrow v$ and we consider balls $B_i$ disjoint from $E_i$ so that $|E_i\cup B_i|=|E_i|+|B_i|$. Then $I_C(v)\le P(E_i\cup B_i)=I_C(v_i)+P(B_i)\le I_C(v)+P(B_i)$ by the monotonicity of $I_C$, and the left-continuity follows by taking limits since $\lim_{i\to\infty}P(B_i)=0$.
\end{proof}


\begin{lemma}
\label{lem:Lip CxR^k}
Let $\{K_i\}_{i\in\nn}$ be a sequence of $m$-dimensional convex bodies converging to a convex body $K$ in Hausdorff distance. Then $\{K_i\times\rr^q\}_{i\in\nn}$ converges to $K\times\rr^q$ in lipschitz distance.
\end{lemma}

\begin{proof}
By \cite[Theorem 3.4]{rv1}, there exists a sequence of bilipschitz maps $f_i:K_i\to K$ such that $\Lip(f_i),\Lip(f_i^{-1})\to 1$ as $i\to \infty$. For every $i\in\nn$, define $F_i:K_i\times\rr^q \to K\times\rr^q$ by
\begin{equation}
\label{eq:Lip CxR^k eks1}
F_i(x,y)=(f_i(x),y),\qquad(x,y)\in K_i\times\rr^q.
\end{equation}
Take now $(x_1,y_1), (x_2,y_2)\in K_i\times\rr^q$. We have
\begin{equation}
\label{eq:Lip CxR^k eks2}
\begin{split}
|F_i(x_1,y_1)-F_i(x_2,y_2)|^2&=|f_i(x_1)-f_i(x_2)|^2+|y_1-y_2|^2
\\
&\le \max\{\Lip(f_i)^2,1 \}\big(|x_1-x_2|^2+|y_1-y_2|^2\big)
\\
&=\max\{\Lip(f_i)^2,1 \}\big|(x_1,y_1)-(x_2,y_2)\big|^2,
\end{split}
\end{equation}
where $|\cdot|$ is the Euclidean norm in the suitable Euclidean space. Hence we get
\[
\limsup_{i\to \infty}\Lip(F_i)\le 1
\]
since $\lim_{i\to \infty}\Lip(f_i)=1$. In a similar way we find  $\limsup_{i\to\infty}\Lip(F_i^{-1})\le 1$. By Remark \ref{rem:lipcomp}, we get $\Lip(F_i^{-1})\Lip(F_i)\ge 1$ and the proof follows.
\end{proof}

\begin{proposition}
\label{prp:Concavdeperfcyldro}
Let $K\subset \rr^m$ be a convex body and $C=K \times \rr^q$. Then $I_C^{(n+1)/n}$ is a concave function. This implies that $I_C$ is concave and every isoperimetric set in $C$ is connected.
\end{proposition}

\begin{proof}
When the boundary of a convex cylinder $C$ is smooth, its isoperimetric profile $I_C$ and its power $I_C^{(n+1)/n}$ are known to be concave using a suitable deformation of an isoperimetric region and the first and second variations of perimeter and volume, as in Kuwert \cite{MR2008339}.

By approximation \cite{sch}, there exists a sequence $\{K_i\}_{i\in\nn}$ of convex bodies in $\rr^m$ with $C^\infty$ boundary such that $K_i\to K$ in Hausdorff distance. Set $C_i=K_i\times\rr^q$. By Lemma \ref{lem:Lip CxR^k}, $C_i\to C$ in lipschitz distance.  Fix now some $v>0$. By Proposition~\ref{prp:exicyl}, there is a sequence of isoperimetric sets $E_i\subset C_i$ of volume $v$.  Thus arguing as in \cite[Theorem 4.1]{rv1}, using the continuity of the isoperimetric profile $I_C$, we get
\begin{equation*}
I_C(v)\le\liminf_{i\to\infty} I_{C_i}(v).
\end{equation*}
Again by Proposition \ref{prp:exicyl} there exists an isoperimetric set $E\subset C$  of volume $v$. Arguing again as in  \cite[Theorem 4.1]{rv1}, we obtain
\begin{equation*}
I_C(v)\ge\limsup_{i\to\infty} I_{C_i}(v).
\end{equation*}
Combining both inequalities we get
\begin{equation*}
I_C(v)=\lim_{i\to\infty} I_{C_i}(v).
\end{equation*}
So $I_C^{(n+1)/n}$, $I_C$ are concave functions as they are pointwise limits of concave functions.

Connectedness of isoperimetric regions is a consequence of the concavity of $I_C^{(n+1)/n}$ as in \cite[Theorem 4.6]{rv1}.
\end{proof}

Assume now that the cylinder $C=K\times\rr^q$ has $C^{2,\alpha}$ boundary. By Theorem~2.6 in Stredulinsky and Ziemer \cite{st-zi}, a local minimizer of perimeter under a volume constraint has the property that either $\cl{\ptl E\cap\intt(C)}$, the closure of $\ptl E\cap\intt(C)$, is either connected or it consists of a union of parallel (totally geodesic) components meeting $\ptl C$ orthogonally with the part of $C$ lying between any two of such components consisting of a \emph{right} cylinder. By the connectedness of isoperimetric regions proven in Proposition~\ref{prp:Concavdeperfcyldro}, $E$ must be a slab in $K\times\rr$. So we have proven the following

\begin{theorem}
\label{thm:sz}
Let $C=K\times\rr^q$ be a convex cylinder with $C^{2,\alpha}$ boundary, and $E\subset C$ an isoperimetric region. Then either the closure of $\ptl E\cap\intt(C)$ is connected or $E$ is an slab in $K\times\rr$.
\end{theorem}

Let us consider now the isoperimetric profile for small volumes. The following is inspired by \cite[Theorem 6.6]{rv1}, although we have simplified the proof.

\begin{theorem}
\label{thm:optinsmalvol}
Let $C=K\times\rr^q$, where $K\subset\rr^m$ is a convex body. Then, after translation, isoperimetric regions of small volume are close to points with the narrowest tangent cone. Furthermore,
\begin{equation}
\label{eq:limvto0}
\lim_{v\to 0}\frac{I_C(v)}{I_{C_{\min}}(v)}=1.
\end{equation}
\end{theorem}

\begin{proof}
To prove \eqref{eq:limvto0},  consider a sequence $\{E_i\}_{i\in\nn}\subset C$ of isoperimetric regions of volumes $v_i\to 0$. By Proposition~\ref{prp:Concavdeperfcyldro}, the sets $ E_i$ are connected.
The key of the proof is to show
\begin{equation}
\label{eq:claim diamto00}
\diam(E_i)\to 0.
\end{equation}
To accomplish this we consider $\la_i\to \infty$ so that the isoperimetric regions $\la_iE\subset \la_i C$ have volume 1. Then we argue exactly as in \cite[Theorem 6.6]{rv1}. We first produce an elimination Lemma as in \cite[Theorem 5.5]{rv1}, with $\eps>0$ independent of $\la_i$, that yields a perimeter lower density bound \cite[Corollary 5.8]{rv1} independent of $\la_i$. Hence the sequence $\{\diam(\la_i E_i)\}_{i\in\nn}$ must be bounded, since otherwise applying the perimeter lower density bound we would get $\pp_{\la_i C}(\la_i E_i)\to \infty$, contradicting Proposition~\ref{prp:ICleICmin}. Since $\{\diam(\la_i E_i)\}_{i\in\nn}$  is bounded, \eqref{eq:claim diamto00} follows.

Translating each set of the sequence $\{E_i\}_{i\in\nn}$, and eventually $C$, we may assume that $E_i$ converges to $0\in\ptl K\times\rr^k$ in Hausdorff distance. Taking $r_i=(\diam(E_i))^{1/2}$ we have $\diam(r_i^{-1}E_i)\to 0$ and so
\begin{equation}
\label{eq:konos}
r_i^{-1}E_i\to 0\quad \text{in Hausdorff distance}.
\end{equation}

Let $q\in\intt(K\cap \cld(0,1))$ and let $D_q$ be an $m$-dimensional closed ball centered at $q$ and contained in $\intt(K\cap\cld(0,1))$. As the sequence $r_i^{-1}K\cap\cld(0,1)$ converges to $K_0\cap\cld(0,1)$ in Hausdorff distance, we construct, using \cite[Thm.~3.4]{rv1}, a family of bilipschitz maps $f_i:r_i^{-1}K\cap \cld(0,1)\to K_0\cap\clb(0,1)$ with $\text{Lip}(f_i)$, $\text{Lip}(f_i^{-1})\to 1$, where $f_i$ is the identity on $D_q$ and is extended linearly along the segments leaving from $q$. We define, as in Lemma~\ref{lem:Lip CxR^k}, the maps $F_i:(r_i^{-1}K\cap\cld(0,1))\times\rr^k\to (K_0\cap\cld(0,1))\times\rr^k$ by $F_i(x,y)=(f_i(x),y)$. These maps satisfy $\Lip(F_i),\,\Lip(F_i^{-1})\to 1$. Since \eqref{eq:konos} holds, the maps $F_i$ have the additional property
\begin{equation}
\label{eq:percone}
\pp_{C_0}(F_i(r_i^{-1}E_i))=\pp_{C_0\cap\clb(0,1)}(F_i(r_i^{-1}E_i)),\qquad\text{for large}\ i\in\nn.
\end{equation}

Thus by  Lemma~\ref{lem:bilip} and \eqref{eq:isopsolang}  we get
\begin{equation}
\label{eq:mincone1}
\begin{split}
\frac{\pp_{C}(E_i)}{\vol{E_i}^{n/(n+1)}}&=\frac{\pp_{r_i^{-1}C}(r_i^{-1}E_i)}{\vol{r_i^{-1}E_i}^{n/(n+1)}}
\\
&\ge \frac{\pp_{C_0}(F_i(r_i^{-1}E_i))}{\vol{F_i(r_i^{-1}E_i)}^{n/(n+1)}}\,(\Lip(F_i)\Lip(F_i^{-1}))^{-n}
\\
&\ge{\alpha(C_0)}^{1/(n+1)}\,(n+1)^{n/(n+1)}\,(\Lip(F_i)\Lip(F_i^{-1}))^{-n}
\end{split}
\end{equation}
Since $E_i$ are isoperimetric regions of volumes $v_i$, passing to the limit we get
\[
\liminf_{i\to\infty}\frac{I_C(v_i)}{v_i^{n/(n+1)}}\ge {\alpha(C_0)}^{1/(n+1)}\,(n+1)^{n/(n+1)}.
\]
From \eqref{eq:isopsolang} we obtain,
\[
\liminf_{i\to\infty}\frac{I_C(v_i)}{ I_{ C_0}(v_i)}\ge 1.
\]
Combining this with \eqref{eq:ICleICmin} and the minimal property of $I_{C_{\min}}$ we deduce
\begin{equation*}
\limsup_{i\to\infty}\frac{I_C(v_i)}{ I_{ C_0}(v_i)}\le
\limsup_{i\to\infty}\frac{I_C(v_i)}{I_{C_{\min}}(v_i)}\le 
1\le \liminf_{i\to\infty}\frac{I_C(v_i)}{ I_{ C_0}(v_i)} .
\end{equation*}
Thus
\begin{equation}
\label{eq:optinsmalvol3}
\lim_{i\to\infty}\frac{I_C(v_i)}{ I_{C_{\min}}(v_i)}=1.
\end{equation}
By \eqref{eq:isopsolang}, we conclude that $C_0$ has minimum solid angle. 
\end{proof}

A convex prism $\Pi$ is a set of the form $P\times\rr^q$  where $P\subset\rr^m$ is a polytope. For convex prisms we are able to characterize the isoperimetric regions for small volumes.

\begin{theorem}
\label{thm:polytops}
Let $\Pi \subset \rr^{n+1}$ be a convex prism.  For small volumes the isoperimetric regions in $\Pi$ are geodesic balls centered at vertices with the smallest solid angle.
\end{theorem}

\begin{proof}
Let $\{E_i\}_{i\in\nn}$ be a sequence of isoperimetric regions in $\Pi$ with $\vol{E_i}\to 0$. By Theorem~\ref{thm:optinsmalvol}, after translation, a subsequence of $E_i$ is close to some vertex $x$ in $\Pi$. Since $\diam(E_i)\to 0$ we can assume that the sets $E_i$ are also subsets of the tangent cone $\Pi_x$ and they are isoperimetric regions in $\Pi_x$. By \cite{FI} the only isoperimetric regions in this cone are, after translation, the geodesic balls centered at $x$. These geodesic balls are also subsets of $\Pi$.
\end{proof} 

To end this section, let us characterize the isoperimetric regions for large volume in the right cylinder $K\times\rr$. We closely follow the proof by Duzaar and Steffen \cite{du-st}, which is slightly simplified by the use of Steiner symmetrization. The case of the cylinder $K\times\rr^q$, with $q>1$, is more involved and will be treated in a different paper.

We shall say that a set $E\subset K\times\rr$ is \emph{normalized} if, for every $x\in K$, the intersection $E\cap (\{x\}\times\rr)$ is a segment with midpoint $(x,0)$.





\begin{theorem}
\label{thm:isocyl}
Let $C=K\times\rr$, where $K\subset\rr^n$ is a convex body. Then there is a constant $v_0>0$ so that the slabs $K\times I$, where $I\subset\rr$ is a compact interval, are the only isoperimetric regions of volume larger than or equal to $v_0$. In particular,  $I_C(v)=2\hh^n(K)$ for all $v\ge v_0$.
\end{theorem}

\begin{proof}
The proof is modeled on \cite[Prop~2.11]{du-st}. By comparison with slabs we have $I_C(v)\le 2\,H^n(K)$ for all $v>v_0$.

Let us assume first that $E\subset K\times\rr$ is a normalized set of finite volume and $H^n(\ptl_C E)\le 2\,H^n(K)$, and let $E^*$ be its orthogonal projection over $K_0=K\times\{0\}$. We claim that, it $H^n(K_0\setminus E^*)>0$, then there is a constant $c>0$ so that
\begin{equation}
\label{eq:importante}
H^n(\ptl_C E)\ge c \vol{E}.
\end{equation}
For $t\in\rr$, we define $E_t=E\cap (K\times\{t\})$. As $E$ is normalized, we can choose $\tau>0$ so that $H^n(E_t)\le H^n(K)/2$ for $t\ge \tau$ and $H^n(E_t)> H^n(K)/2$ for $0<t<\tau$.

For $t\ge\tau$ we apply the coarea formula and Lemma~\ref{lem:I_C(v)> cv} to get
\begin{equation}
\label{eq:mitad}
\begin{split}
H^n(\ptl_C E)&\ge H^n(\ptl_C E\cap (K\times [t,\infty))
\\
&\ge \int_{\tau}^{+\infty}\h^{n-1} (\ptl_C{E_s})\,ds\ge c_1\int_\tau^{+\infty} H^n(E_s)\,ds\ge c_1\vol{E\cap (K\times [\tau,+\infty))},
\end{split}
\end{equation}
where $c_1$ is a constant only depending on $H^n(K)/2$.

Let $S_t=K\times\{t\}$. For $0<t<\tau$ we have
\begin{equation}
\label{eq:st-et}
\h^{n} (S_t \setminus E_t)\ge \h^{n} (\ptl_C E \cap (K\times (0,t))),
\end{equation}
since otherwise
\begin{equation*}
\label{eq:2a}
\begin{split}
H^n(K)&= \h^{n} (S_t\setminus E_t)+ \h^{n} (E_t)
\\
&< \h^{n} (\ptl_C E \cap  (K\times(0,t)))+ \h^{n} (\ptl_C E \cap  (K\times [t,+\infty)))
\\
&\le\h^{n} (\ptl_C E)/2,
\end{split}
\end{equation*}
and we should get a contradiction to our assumption $\hh^n(\ptl_C E)\le 2\,H^n(K)$, what proves \eqref{eq:st-et}. So we obtain from \eqref{eq:st-et} and Lemma \ref{lem:I_C(v)> cv}
\begin{equation}
\label{eq:st-etestimate}
\begin{split}
\h^{n} (S_t\setminus E_t)&\ge \h^{n} (\ptl_C E  \cap (K\times (0,t)))
\\
&\ge \int_0^{t}\h^{n-1} (\ptl_C{E\cap S_t})dt
\\
&\ge c_2 \int_0^{\tau}\h^{n} (S_t \setminus E_t)^{(n-1)/n}dt,
\end{split}
\end{equation}
where $c_2$ is a constant only depending on $H^n(K)/2$. Letting $y(t)=\h^{n} (S_t \setminus E_t)$, inequality \eqref{eq:st-etestimate} can be rewritten as the integral inequality
\begin{equation*}
y(t)\ge c_2\int_0^{t}y(s)^{(n-1)/n}ds.
\end{equation*}
Since $H^n(K_0\setminus E^*)>0$ by assumption and $E$ is normalized, we have $y(t)>0$ for all $t>0$, and so
\[
2\,H^n(K)\ge \hh^n(S_\tau\setminus E_\tau)=y(\tau)\ge \frac{c_2^n}{n^n}\,\tau^n,
\]
what implies
\begin{equation}
\tau\le\frac{n}{c_2\,(2\,H^n(K))^{1/n}}.
\end{equation}
We finally estimate
\begin{equation}
\label{eq:mitad2}
\vol{E \cap (K\times {[0,\tau])}}=\int_0^{\tau}\h^{n} (E_t)\,dt
\le 2 \h^{n} ({E_0})\,\tau\le \frac{n}{c_2\,(2\,H^n(K))^{1/n}} \h^{n} (\ptl_C E ).
\end{equation}

Combining \eqref{eq:mitad} and \eqref{eq:mitad2}, we get \eqref{eq:importante}. This proves the claim.

Let now $E\subset K\times\rr$ be an isoperimetric region of large enough volume $v$. Following Talenti \cite{talenti} or Maggi \cite{MR2976521}, we may consider its Steiner symmetrized $\text{sym}\,E$. The set $\text{sym}\,E$ is normalized and we have $\vol{E}=\vol{\text{sym}\,E}$ and $P_C(\text{sym}\,E)\le P_C(E)$. Of course, since $E$ is an isoperimetric region we have $P_C(\text{sym}\,E)=P_C(E)$. If $H^n(K_0\setminus E^*)>0$, then \eqref{eq:importante} implies
\[
P_C(E)=P_C(\text{sym}\,E)=H^n(\ptl _C(\text{sym}\,E))\ge c\,\vol{\text{sym}\,E}=c\,\vol{E},
\]
providing a contradiction since $I_C\le 2\,H^n(K)$.

We conclude that $H^n(K_0\setminus E^*)=0$ and that $E$ is the intersection of the subgraph of a function $u:K\to\rr$ and the epigraph of a function $v:K\to\rr$. The perimeter of $E$ is then given by
\[
\pp_C(E)=\int_K\sqrt{1+|\nabla u|^2}\,d\h^n+\int_K\sqrt{1+|\nabla v|^2}\,d\h^n  \ge 2\h^n(K),
\]
with equality if and only if $\nabla u=\nabla v=0$. Hence $u$, $v$ are constant functions and $E$ is a slab.
\end{proof}

As a consequence we have

\begin{corollary}
\label{cor:half-cylinder}
Let $K\subset\rr^n$ be a convex body and  $C= K\times[0,\infty)$. Then there is a constant $v_0>0$ such that any isoperimetric region in $M$ with volume $v\ge v_0$ is the slab $K\times [0,b]$, where $b=v/H^n(K)$. In particular, $I_C(v)=H^n(K)$ for $v\ge v_0$.
\end{corollary}

\begin{proof}
Just reflect with respect to the plane $x_{n+1}=0$ and apply Theorem~\ref{thm:isocyl}. Alternatively, the proof of Theorem~\ref{thm:isocyl} can also be adapted to handle this case.
\end{proof}

\section{Cilindrically bounded convex sets}
\label{sec:cylindrically}
Given a cylindrically bounded convex body $C\subset\rr^{n}\times\rr$ so that $K$ is the closure of the orthogonal projection of $C$ over $\rr^n\times\{0\}$, we shall say that $C_\infty=K\times\rr$ is the \emph{asymptotic cylinder} of $C$. Recall that, from our definition, $C$ is different from a cylinder. Assuming $C$ is unbounded in the positive vertical direction, the asymptotic cylinder can be obtained as a Hausdorff limit of downward translations of $C$. Another property of $C_\infty$ is the following: given $t\in\rr$, define
\begin{equation}
\label{eq:C_t}
C_t=C\cap (\rr^n\times\{t\}).
\end{equation}
Then the orthogonal projection of $C_t$ to $\rr^n\times\{0\}$ converges in Hausdorff distance to the basis $K$ of the asymptotic cylinder when $t\uparrow +\infty$ by \cite[Thm.~1.8.16]{sch}. In particular, this implies
\[
\lim_{t\to +\infty} H^n(C_t)=H^n(K).
\]

Let us prove now that the isoperimetric profile of $C$ is asymptotic to the one of the half-cylinder

\begin{theorem}
\label{thm:asym-profile}
Let $C\subset\rr^{n+1}$ be a cylindrically bounded convex body with asymptotic cylinder $C_\infty=K\times\rr$. Then
\begin{equation}
\label{eq:asymp-profile}
\lim_{v\to \infty} I_C(v)=\h^n(K).
\end{equation}
\end{theorem}
\begin{proof}
We assume that $C$ is unbounded in the positive $x_{n+1}$-direction and consider the sets $\Om(v)=C\cap (\rr^n\times (-\infty,t(v)])$, where $t(v)$ is chosen so that $\vol{\Om(v)}=v$. Then
\begin{equation*}
I_C(v)\le \pp_C(\Om (v))\le \h^n(K),
\end{equation*}
and taking limits we get
\[
\limsup_{v\to \infty} I_C(v)\le \h^n(K).
\]

Let us prove now that
\begin{equation}
\label{eq:liminfhk}
\hh^n(K)\le\liminf_{v\to\infty} I_C(v).
\end{equation}
Fix $\eps>0$. We consider a sequence of volumes $v_i\to\infty$ and a sequence $E_i\subset C$ of  finite perimeter sets of volume $v_i$ with smooth boundary, so that 
\begin{equation}
\label{eq:defei}
\pp_C(E_i)\le I_C(v_i) +\eps.
\end{equation}
We shall consider two cases. Recall that $(E_i)_t=E_i\cap (\rr^n\times\{t\})$.

\noindent Case 1. $\displaystyle\liminf_{i\to \infty}\big(\sup_{t>0}\ \h^n((E_i)_t)\big)=\h^n(K)$.

\noindent This is an easy case. Since the projection over the horizontal hyperplane does not increase perimeter we get
\[
I_C(v_i)+\eps\ge\pp_C(E_i)\ge \sup_{t>0}\ \h^n((E_i)_t).
\]
Taking inferior limit, we get \eqref{eq:liminfhk} since $\eps >0$ is arbitrary.

\noindent Case 2. $\displaystyle\liminf_{i\to \infty}\big(\sup_{t>0}\ \h^n((E_i)_t)\big)<\h^n(K)$.

\noindent In this case, passing to a subsequence, there exists $v_0<\h^n(K)$ such that $\h^n((E_i)_t)\le v_0$ for all $t$. By \cite[Thm.~1.8.16]{sch} we have $\h^n(C_t)\to \h^n(K)$. Hence there exists $t_0>0$ such that $v_0<\h^n(C_t)$ for $t\ge t_0$. By Lemma \ref{lem:I_C(v)> cv}, for $c_t=I_{C_t}(v_0)/v_0$, we get
\[
I_{C_t}(v)\ge c_t v,\ \text{for all}\ v\le v_0,\ t\ge t_0.
\]
Furthermore, as $I_{C_t}(v_0) \to I_{K}(v_0)>0$ and $I_K(v_0)>0$, we obtain the existence of $c>0$ such that $c_t>c$ for $t$ large enough. Taking $t_0$ larger if necessary we may assume $c_t>c$ holds when $t\ge t_0$. Thus for large $i\in \nn$ we obtain
\begin{align*}
\vol{E_i}&=\int_0^{\infty}\h^{n}((E_i)_t)\,dt\le b+\int_{t_0}^{\infty}\h^{n}((E_i)_t)\,dt
\\
&\le b+\int_{t_0}^{\infty}c_t^{-1}\h^{n-1}((\ptl E_i)_t)\,dt
\\
&\le b+c^{-1}\int_{0}^{\infty}\h^{n-1}((\ptl E_i)_t)\,dt\le b+ c^{-1}\pp_C(E_i),
\end{align*}
where $b=t_0 \h^{n}(K)$.
So $\pp_C(E_i)\to \infty$ when  $\vol{E_i}\to \infty$. From \eqref{eq:defei} and $I_{C}\le \h^{n}(K)$ we get a contradiction. This proves that Case 2 cannot hold and so \eqref{eq:liminfhk} is proven.
\end{proof}

Let us show now that the isoperimetric profile of $C$ is continuous and, when the boundary of $C$ is smooth enough, that the isoperimetric profile $I_C$ and its normalization $I_C^{(n+1)/n}$ are both concave non-decreasing functions. We shall need first some preliminary results.

\begin{proposition}
\label{prp: asympt ineq C}
Let $C\subset\rr^{n+1}$ be a cylindrically  bounded convex set, and $C_{\infty}=K\times \rr$ its asymptotic cylinder. Consider a diverging sequence of finite perimeter sets $\{E_i\}_{i\in\nn}\subset C$ such that $v=\lim_{i\to\infty}\vol{E_i}$. Then
\[
\liminf_{i\to\infty} P_C(E_i)\ge I_{C_{\infty}}(v) .
\]
\end{proposition}

\begin{proof}
Without loss of generality we assume $E_i\subset C\cap \{ x_{n+1}\ge i \}$. Let $r>0$ and $t_0>0$ so that the half-cylinder $B(0,r)\times [t_0,+\infty)$ is contained in $C\cap\{x_{n+1}\ge t_0\}$. Consider the horizontal sections $C_t=C\cap\{x_{n+1}=t\}$, $(C_\infty)_t=C_\infty\cap\{x_{n+1}=t\}$. We define a map $F:C\cap\{x_{n+1}\ge t_0\}\to C_\infty\cap\{x_{n+1}\ge t_0\}$ by
\[
F(x,t)=(f_t(x),t),
\]
where $f_t:C_t\to (C_\infty)_t$ is defined as in (3.6) in \cite{rv1}. For $i\in\nn$, let $F_i=F|_{C\cap\{x_{n+1}\ge i\}}$. We will check that $\max\{\Lip(F_i), \Lip(F_i^{-1})\}\to 1$ when $i\to \infty$.

Take now $(x,t)$, $(y,s)\in C\cap\{x_{n+1}\ge i\}$, and assume $t\ge s$, $i\ge t_0$. Then we have
\begin{equation}
\label{eq:fxt}
\begin{split}
|F(x,t)-F(y,s)|&=\big(|f_t(x)-f_s(y)|^2+|t-s|^2\big)^{1/2}
\\
&=\big(|f_t(x)-f_t(y)+f_t(y)-f_s(y)|^2+|t-s|^2\big)^{1/2}
\\
&=\big(|f_t(x)-f_t(y)|^2+|f_t(y)-f_s(y)|^2
\\
&\qquad\qquad\qquad+2\,|f_t(x)-f_t(y)||f_t(y)-f_s(y)|+|t-s|^2\big)^{1/2}
\end{split}
\end{equation}
We have $|(f_t(x)-f_t(y))|\le\Lip(f_t)|x-y|$. By \cite[Theorem~3.4]{rv1}, we can write $\Lip(f_t)<(1+\eps_i)$ for $t\ge i$, where $\eps_i\to 0$ when $i\to\infty$. Hence
\begin{equation}
\label{eq:eksis f t-f s < Lip conical }
|(f_t(x)-f_t(y))|\le(1+\eps_i)\,|x-y|, \qquad\text{for}\  t\ge i.
\end{equation}
We estimate now $|f_t(y)-f_s(y)|$. In case $|y|\le r$, we trivially have $|f_t(y)-f_s(y)|=0$. So we assume $|y|\ge r$. For $u\in\esf^{n-1}$, consider the functions $\rho_t(u)=\rho(C_t,u)$, $\rho(u)=\rho (K,u)$.  Observe that, for every $u\in \esf^{n}$ orthogonal to $\ptl/\ptl x_{n+1}$, the 2-dimensional half-plane defined by $u$ and $\ptl/\ptl x_{n+1}$ intersected with $C$ is a 2-dimensional convex set, and the function $t\mapsto \rho_t(u) $ is concave with a horizontal asymptotic line at height $\rho(u)$. So we have, taking $u=y/|y|$,
\begin{equation*}
\frac{|f_t(y)-f_s(y)|}{|t-s|}=\frac{\big(|y|-r\big)}{|t-s|} \Big|\frac{\rho_t(u)-r}{{\rho}(u)-r}
-\frac{{\rho}_s(u)-r}{{\rho}(u)-r}\Big|\le \frac{\big|{\rho}_t(u)-{\rho}_s(u)\big|}{|t-s|},
\end{equation*}
since $|y|-r\ge \rho(u)-r$. Using the concavity of $t\mapsto\rho_t(u)$ we get
\begin{equation*}
\frac{\big|{\rho}_t(u)-{\rho}_s(u)\big|}{|t-s|}\le \big|\rho_i(u)-\rho_{i-1}(u)\big|, \qquad\text{for}\ t,s\ge i.
\end{equation*}
Letting $\ell_i=\sup_{u\in\esf^{n-1}}|\rho_i(u)-\rho_{i-1}(u)|$, we get
\begin{equation}
\label{eq:estimation f_t(y)-f_s(y) final}
|f_t(y)-f_s(y)|\le\ell_i\,|t-s|.
\end{equation}
As $C_{\infty}$ is the asymptotic cylinder of $C$ we conclude that $\ell_i\to 0$ when $i\to \infty$.

From \eqref{eq:fxt}, \eqref{eq:eksis f t-f s < Lip conical }, \eqref{eq:estimation f_t(y)-f_s(y) final}, and trivial estimates, we obtain
\begin{equation}
\label{eq:eksis F(x,t)-F(y,s) telik}
|F_i(x,t)-F_i(y,s)|\le \big((1+\eps_i)^2+\ell_i^2+(1+\eps_i)\,\ell_i\big)^{1/2}\,|(x,t)-(y,s)|
\end{equation}

Now $\eps_i\to 0$ and $\ell_i\to 0$ as $i\to\infty$. Thus inequality \eqref{eq:eksis F(x,t)-F(y,s) telik} yields
\[
\limsup_{i\to\infty}\Lip(F_i)\le 1.
\]

Similarly we find $\limsup_{i\to\infty}\Lip(F_i^{-1})\le 1$ and since $\Lip(F_i^{-1})\Lip(F_i)\ge 1$ by Remark \ref{rem:lipcomp}, we finally get $\max\{\Lip(F_i),\Lip(F_i^{-1})\}\to 1$ when $i\to\infty$. 

Thus we have
\begin{equation}
\label{eq:conthfi}
\begin{split}
v=\lim_{i\to \infty}\vol{E_i}&=\lim_{i\to \infty} \vol{F_i(E_i)},
\\
\liminf_{i\to \infty} \pp_C(E_i)&= \liminf_{i\to \infty} \pp_{C_{\infty}}(F_i(E_i)).
\end{split}
\end{equation}
Now from \eqref{eq:conthfi} and the continuity of $I_{C_{\infty}}$ we get
\[
\liminf_{i\to \infty} \pp_C(E_i)= \liminf_{i\to \infty} \pp_{C_{\infty}}(F_i(E_i))\ge I_{C_{\infty}}(v).
\]
\end{proof}

\begin{lemma}
\label{lem:aproximacion de E por E i en C}
Let $C \subset \rr^{n+1}$ be a cylindrically bounded convex set and $C_{\infty}=K\times \rr$ its asymptotic cylinder. Let $E_\infty\subset C_{\infty} $ a bounded set of finite perimeter. Then there exists a sequence $\{E_i\}_{i\in\nn}\subset C$ of finite perimeter sets such that $\vol{E_i}=\vol{E_\infty}$ and $\lim_{i\to\infty}\pp_C(E_i)=\pp_{C_{\infty}}(E_\infty)$.
\end{lemma}

\begin{proof}
Let $e_{n+1}=(0,\ldots,0,1)\in\rr^{n+1}$. We consider the truncated downward translations of $C$ defined by
\[
C_i =(-i\,e_{n+1}+C)\cap \{t\ge 0\},\ i\in\nn.
\]
These convex bodies have the same asymptotic cylinder and
\begin{equation}
\label{eq:cup C m=C infty}
\bigcup_{i\in\nn} C_i=C_{\infty}\cap [0,\infty).
\end{equation}
Translating $E_\infty$ along the vertical direction if necessary we assume $E_\infty\subset \{t> 0\}$. Consider the sets $G_i=E_\infty\cap C_i$. For large indices $G_i$ is not empty by \eqref{eq:cup C m=C infty}.
By the monotonicity of the Hausdorff measure we have $\vol{G_i} \uparrow  \vol{E_\infty}$, and $ \h^{n}(\ptl G_i\cap \intt(C_i)) \uparrow \h^{n}(\ptl E_\infty \cap \intt(C_\infty))$. As $E_\infty$ is bounded, for large $i$ we can find Euclidean geodesic balls $B_i\subset \intt(C_i)$, disjoint from $G_i$,  such that $\vol{B_i}=\vol{E_\infty}-\vol{G_i}$. Obviously the volume and and the perimeter of these balls go to zero when $i$ goes to infinity. Then $E_i=G_i\cup B_i$ are the desired sets.
\end{proof}

\begin{proposition}
\label{prp:coniclycontis}
Let $C\subset\rr^{n+1}$ be a cylindrically bounded convex body. Then $I_C$ is continuous.
\end{proposition}

\begin{proof}
Let $C_\infty=K\times\rr$ be the asymptotic cylinder of $C$. The continuity of the isoperimetric profile $I_C$ at $v=0$ is proven by comparison with geodesic balls intersected with $C$.

Fix $v>0$ and let $\{v_i\}_{i\in\nn}$ be a sequence of positive numbers converging to $v$. Let us prove first the lower semicontinuity of $I_C$. By the definition of isoperimetric profile, given $\eps>0$, there is a finite perimeter set $E_i$ of volume $v_i$ so that $I_C(v_i)\le\pp_C(E_i)\le I_C(v_i)+\tfrac{1}{i}$, for every $i\in\nn$. Reasoning as in \cite[Thm.~2.1]{r-r}, we can decompose $E_i=E_i^c\cup E_i^d$ into convergent and diverging pieces, and there is a finite perimeter set $E\subset C$, eventually empty, so that
\begin{equation}\label{eq:ena parad}
\begin{split}
\vol{E_i}&=\vol{E_i^c}+\vol{E_i^d},
\\
\pp_C(E_i)&=\pp_C(E_i^c)+\pp_C(E_i^d),
\\
\vol{E_i^c}&\to\vol{E},
\\
\pp_C(E)&\le \liminf_{i\to\infty} \pp_C(E_i^c).
\end{split}
\end{equation}
Let $w_1=\vol{E}$. By Proposition~\ref{prp:exicyl}, there exists an isoperimetric region $E_\infty\subset C_{\infty}$ of volume $\vol{E_\infty}=w_2=v-w_1$. By Proposition~\ref{prp: asympt ineq C} we have $\pp_{C_{\infty}}(E_\infty)\le \liminf_{i\to\infty} \pp_C(E_i^d)$. Hence
\begin{align*}
I_C(v)\le I_C(w_1)+I_{C_{\infty}}(w_2)&\le\pp_C(E)+\pp_{C_{\infty}}(E_\infty)
\\
&\le\liminf_{i\to\infty} \pp_C(E_i^c)+\liminf_{i\to\infty} \pp_C(E_i^d)
\\
&\le\liminf_{i\to\infty} \pp_C(E_i)
\\
&=\liminf_{i\to\infty} I_C(v_i).
\end{align*}
To prove the upper semicontinuity of $I_C$ we will use a standard  variational argument. Fix $\eps>0$. 
We can find a bounded set $E\subset C$ of volume $v$ with $I_C(v)\le \pp_C(E)\le I_C(v)+ \eps$ and a smooth open portion $U\subset \ptl_{C} E$ contained in the relative boundary. We construct a variation compactly supported in $U$ of $E$ by sets $E_s$ so that $\vol{E_s}=v+s$ for $s\in (-\delta,\delta)$. Then there is $M>0$ so that
\[
|\h^{n}(\ptl_C E_s)-\h^{n}(\ptl_C E)|\le M\,|\vol{E_s}-\vol{E}|.
\]
Hence
\begin{align*}
I_C(v+s)&\le \hh^n(\ptl_C E_s)\le \hh^n(\ptl_C E)
\\
&\le I_C(v)+\eps +M\,\big(\vol{E_s}-\vol{E}\big)
\\
&=I_C(v)+\eps+Ms.
\end{align*}
Taking a sequence $v_i\to v$  we get $\limsup_{i\to\infty}I_C(v_i)\le I_C((v)+\eps$. As $\eps$ is arbitrary we obtain the upper semicontinuity of $I_C$.
\end{proof}

\begin{proposition}
\label{prp:excylindical}
Let $C\subset\rr^{n+1}$ be a cylindrically bounded convex body with  asymptotic cylinder $C_\infty=K\times\rr$. Assume that both $C$ and $C_\infty$ have smooth boundary. Then isoperimetric regions exist on $C$ for large volumes and have connected boundary. Moreover $I_C^{(n+1)/n}$ and so $I_C$ are concave non-decreasing functions. 
\end{proposition}

\begin{proof}
Fix $v>0$. By \cite[Thm.~2.1]{r-r}  there exists an isoperimetric region $E\subset C$ (eventually empty) of volume $\vol{E}=v_1\le v$, and a diverging sequence $\{E_i\}_{i\in\nn}$ of finite perimeter sets of volume $v_2=v-v_1$, such that
\begin{equation}
\label{eq.parestab1}
I_C(v)=\pp_C(E)+\lim_{i\to\infty}\pp_C(E_i)
\end{equation}

By Proposition \ref{prp:exicyl}, there is an isoperimetric region $E_\infty\subset C_\infty$ of volume~$v_2$. We claim
\begin{equation}
\label{eq.parestab2}
\lim_{i\to\infty}\pp_C(E_i)= \pp_{C_{\infty}}(E_\infty).
\end{equation}
If \eqref{eq.parestab2} does not hold, then Proposition~\ref{prp: asympt ineq C} implies $\liminf_{i\to\infty} \pp_C(E_i)> I_{C_{\infty}}(v_2)$, and Lemma \ref{lem:aproximacion de E por E i en C} provides a sequence of finite perimeter sets in $C$, of volume $v_2$, approaching $E_\infty$. This way we can build a minimizing sequence of sets of volume $v$ whose perimeters converge to some quantity strictly smaller than $I_C(v)$, a contradiction that proves \eqref{eq.parestab2}. From \eqref{eq.parestab1} and \eqref{eq.parestab2} we get
\begin{equation}
\label{eq.parestab3}
I_C(v)=P_C(E)+P_{C_\infty}(E_\infty).
\end{equation}

Reasoning as in the proof of Theorem~2.8 in \cite{MR1883725}, the configuration $E\cup E_\infty$ in the disjoint union of the sets $C$, $C_\infty$ must be stationary and stable, since otherwise we could slightly perturb $E\cup E_\infty$, keeping constant the total volume, to get a set ${E}'\cup {E'_\infty}$ such that
\begin{equation*}
\pp_C(E')+\pp_{C_{\infty}}(E'_\infty)<\pp_C(E)+\pp_{C_{\infty}}({E_\infty}),
\end{equation*}
contradicting \eqref{eq.parestab3}.




Now as $C, C_{\infty}$ are convex and have smooth boundary, we can use a stability argument similar to that in \cite[Proposition 3.9]{bay-rosal}  to conclude that one of the sets $E$ or $E_\infty$ must be empty and the remaining one must have connected boundary. A third possibility, that $\ptl_C E \cup\ptl_{C_{\infty}} E_\infty$ consists of a finite number of hyperplanes intersecting orthogonally both $C$ and $C_{\infty}$, can be discarded since in this case $E_\infty$ would be a slab with $P_{C_\infty}(E_\infty)=2 H^n(K)>I_C$.

If $v$ is large enough so that isoperimetric regions in $C_\infty$ are slabs, then the above argument shows existence of isoperimetric regions of volume $v$ in $C$. 

As $I_C$ is always realized by an isoperimetric set in $C$ or $C_{\infty}$, the arguments in \cite[Theorem 3.2]{bay-rosal} imply that the second lower derivative of $I_C^{(n+1)/n}$ is non-negative. As $I_C^{(n+1)/n}$ is continuous by Proposition~\ref{prp:coniclycontis}, Lemma~3.2 in \cite{MR1803220} implies that $I_C^{(n+1)/n}$ is concave and hence non-decreasing. Then $I_C$ is also concave as a composition of $I_C^{(n+1)/n}$ with the concave non-increasing function $x\mapsto x^{n/(n+1)}$.

The connectedness of the isoperimetric regions in $C$ follows easily as an application of the concavity of $I_C^{(n+1)/n}$, as in \cite[Theorem 4.6]{rv1}.
\end{proof}
The concavity of $I_C^{(n+1)/n}$ also implies the following Lemma. The proof in \cite[Lemma 4.9]{rv1} for convex bodies also holds in our setting.
\begin{lemma}
\label{lem:_IC<la C}
Let $C$ be be a cylindrically bounded convex body with  asymptotic cylinder $C_\infty$. Assume that both $C$ and $C_\infty$ have smooth boundary. Let $\la\ge 1$. Then
\begin{equation}
\label{eq:ilacic}
I_{\la C}(v)\ge I_C(v)
\end{equation}
for all $0\le v\le\vol{C}$.
\end{lemma}

Our aim now is to get a density estimate for isoperimetric regions of large volume in Theorem~\ref{thm:leon rigot lem 42cyl}. This estimate would imply the convergence of the free boundaries of large isoperimetric regions to hyperplanes in Hausdorff distance given in Theorem~\ref{thm:thmcyli}.

\begin{proposition}
\label{prp:cylrelativ}
\mbox{}
Let $C$ be cylindrically bounded convex body with asymptotic cylinder $C_{\infty}$. Given $r_0>0$, there exist positive constants $M$, $\ell_1$, only depending on $r_0$ and $C$,  $C_{\infty}$, and a universal positive constant $\ell_2$ so that
\begin{equation}
\label{eq:isnqgdbl1}
\pp_{\clb_C(x,r)}(v)\ge M\, {\min \{v,\vol{\clb_C(x,r)}-v\}}^{n/(n+1)},\end{equation}
for all $x\in C$, $0<r\le r_0$, and $0<v<\vol{\clb(x,r)}$. Moreover
\begin{equation}
\label{eq:isnqgdbl1a}
\ell_1 r^{n+1} \le \vol{\clb_C(x,r)} \le \ell_2 r^{n+1},
\end{equation}
for any $x\in C$, $0<r\le r_0$.
\end{proposition}
\begin{proof}
Reasoning as in \cite[Theorem 4.12]{rv1}, it is enough to show
\begin{equation*}
\Lambda_0=\inf_{x\in C}\inr(\clb_C(x,r_0))>0.
\end{equation*}
To see this consider a sequence $\{x_i\}_{i\in\nn}$ so that $\inr(\clb_C(x_i,r_0))$ converges to $\Lambda_0$. If $\{x_i\}_{i\in\nn}$ contains a bounded subsequence then we can extract a convergent subsequence to some point $x_0\in C$ so that $\Lambda_0=\inr(\clb(x_0,r_0)>0$. If $\{x_i\}_{i\in\nn}$ is unbounded, we translate vertically the balls $\clb_C(x_i,r_0)$ so that the new centers $x_i'$ lie in the hyperplane $x_{n+1}=0$. Passing to a subsequence we may assume that $x_i'$ converges to some point $x_0\in C_\infty$. By the proof of Proposition~\ref{prp: asympt ineq C}, we have Hausdorff convergence of the translated balls to $\clb_{C_\infty}(x_0,r_0)$ and so $\Lambda_0=\inr(\clb_{C_\infty}(x_0,r_0))>0$.
\end{proof}

The next Lemma appeared in \cite[Lemma~5.4]{rv1}. We recall the proof here for completeness.

\begin{lemma}
\label{lem:fv}
For any $v>0$, consider the function $f_v:[0,v]\to\rr$ defined by
\begin{equation*}
f_v(s)=s^{-n/(n+1)}\,\bigg(\bigg(\frac{v-s}{v}\bigg)^{n/(n+1)}-1\bigg).
\end{equation*}
Then there is a constant $0<c_2<1$ that does not depends on $v$ so that $f_v(s)\ge -(1/2)\,v^{-n/(n+1)}$ for all $0\le s\le c_2\,v$.
\end{lemma}

\begin{proof}
By continuity, $f_v(0)=0$. Observe that $f_v(v)=-v^{-n/(n+1)}$ and that, for $s\in [0,1]$, we have $f_v(sv)=f_1(s)\,v^{-n/(n+1)}$. The derivative of $f_1$ in the interval $(0,1)$ is given by 
\[
f_1'(s)=\frac{n}{n+1}\,\frac{(s-1)+(1-s)^{n/(n+1)}}{s-1}\,s^{-1-n/(n+1)},
\]
which is strictly negative and so $f_1$ is strictly decreasing. Hence there exists $0<c_2<1$ such that $f_1(s)\ge -1/2$ for all $s\in [0,c_2]$. This implies $f_v(s)=f_1(s/v)\,v^{-n/(n+1)}\ge -(1/2)\,v^{-n/(n+1)}$ for all $s\in [0,c_2 v]$.
\end{proof}
\begin{theorem}
\label{thm:leon rigot lem 42cyl}
\mbox{}
Let $C\subset \rr^{n+1}$ be a cylindrically bounded convex body with asymptotic cylinder $C_\infty=K\times\rr$. Assume that $C$, $C_\infty$ have smooth boundary. Let $E\subset C$ an isoperimetric region of volume $v>1$.  Choose $\eps$ so that
\begin{equation}
\label{eq:epsfine}
0<\eps<\bigg\{{\ell_2}^{-1},c_2,\frac{{\ell_2}^n}{8^{n+1}},{\ell_2}^{-1}\bigg(\frac{I_C(1)}{4}\bigg)^{n+1}\bigg\},
\end{equation}
where $c_2$ is the constant in Lemma~\ref{lem:fv}., and $\ell_1$, $\ell_2$ the constants in Proposition~\ref{prp:cylrelativ}.

Then, for any $x\in C$ and $R\le 1$ so that $h(x,R)\le\eps$, we get
\begin{equation}
\label{eq:hr/2=0}
h(x,R/2)=0.
\end{equation}
Moreover, in case $h(x,R)=\vol{E\cap B_C(x,R)}\vol{B_C(x,R)}^{-1}$, we get $|E\cap B_C(x,R/2)|=0$ and, in case $h(x,R)=\vol{B_C(x,R)\setminus E}\vol{B_C(x,R)}^{-1}$, we have $|B_C(x,R/2)\setminus E|=0$.
\end{theorem}

\begin{proof}
From the concavity of $I_C^{(n+1)/n}$ and the fact that $I_C(0)=0$ we get, as in Lemma~4.9 in \cite{rv1}, the following inequality
\begin{equation}
\label{eq:c1}
I_C(w)\ge c_1w^{n/(n+1)},\qquad c_1=I_C(1),
\end{equation}
for all $0\le w\le 1$.

Assume first that
\[
h(x,R)=\frac{\vol{E\cap B_C(x,R)}}{\vol{B_C(x,R)}}.
\]
Define $m(t)=\vol{E\cap B_C(x,t)}, 0<t\le R$. Thus $m(t)$ is a non-decreasing function. For $t\le R\le 1$ we get
\begin{equation}
\label{eq:m(t)<eps}
m(t)\le m(R)=\vol{E\cap B_C(x,R)}=h(x,R) \,\vol{B_C(x,R)}\le h(x,R)\,\ell_2 R^{n+1}\le\eps\ell_2<1,
\end{equation}
by \eqref{eq:epsfine}. Since $ v>1$, we get $v-m(t)>0$.

By the coarea formula, when $m'(t)$ exists, we obtain
\begin{equation}
\label{eq:le ri cofo}
m'(t)=\frac{d}{dt}\int_0^t\h^{n}(E\cap \ptl_C B(x,s))ds=\h^{n}(E\cap \ptl_C B(x,t)).
\end{equation}
Define
\begin{equation}
\label{eq:def la, E(t)}
\la(t)=\frac{v^{1/{(n+1)}}}{(v-m(t))^{1/{(n+1)}}},\qquad E(t)=\la(t)(E\setminus B_C(x,t)).
\end{equation}
Then $E(t)\subset \la(t)C$ and $\vol{E(t)}=\vol{E}=v$. 
By Lemma \ref{lem:_IC<la C}, we get $I_{\la(t) C}\ge I_C$ 
since $\la(t)\ge 1$. Combining this with \cite[Cor.~5.5.3]{Ziemer}, equation~\eqref{eq:le ri cofo}, and elementary properties of the perimeter functional, we have
\begin{equation}
\label{eq:diffineq}
\begin{split}
I_C(v)&\le I_{\la(t) C}(v)\le \pp_{\la(t)C}(E(t))=\la^n(t)\,\pp_C(E\setminus B_C(x,t))
\\
&\le \la^n(t)\big(\pp_C(E)-\pp(E,B_C(x,t))+\h^{n}(E\cap \ptl B_C(x,t))\big)
\\
&\le \la^n(t)\big(\pp_C(E)-\pp_C(E\cap B_C(x,t))+2\h^{n}(E\cap \ptl B_C(x,t))\big)
\\
&\le \la^n(t)\big(I_C(v)-c_1m(t)^{n/{(n+1)}}+2m'(t)\big),
\end{split}
\end{equation}
where $c_1$ is the constant in \eqref{eq:c1}. Multiplying both sides by $I_C(v)^{-1}{\la(t)}^{-n}$ we find
\begin{equation}
\label{eq:flex..}
{\la(t)}^{-n}-1+\frac{c_1}{I_C(v)}m(t)^{n/{(n+1)}}\le \frac{2}{I_C(v)}m'(t).
\end{equation}
As we have $I_C\le H^n(K)$, and $I_C$ is concave by Proposition \ref{prp:excylindical}, there exists a constant $\alpha>0$ such that $I_C\ge \alpha$ for sufficient large volumes. Set
\begin{equation}
\label{eq:def of a, b}
a=\frac{2}{\alpha}\ge\frac{2}{I_C(v)},\quad\text{and}\quad b=\frac{c_1}{H^n(K)}\le\frac{c_1}{I_C(v)}.
\end{equation}
From the definition \eqref{eq:def la, E(t)} of $\la(t)$ we get
\begin{equation}
\label{eq:f(m(t))le am'(t)}
f(m(t))\le am'(t)\,\qquad\h^1\text{-a.e},
\end{equation}
where
\begin{equation}
\frac{f(s)}{s^{n/(n+1)}}=b+\frac{\big(\frac{v-s}{v}\big)^{n/{(n+1)}}-1}{s^{n/(n+1)}}.
\end{equation}
By Lemma~\ref{lem:fv}, there exists a universal constant $0<c_2<1$, not depending on $v$, so that
\begin{equation}
\label{eq:epsilon le-ri}
\frac{f(s)}{s^{n/{n+1}}}\ge b/2\qquad \text{whenever}\qquad 0< s\le c_2 .
\end{equation}
Since $\eps\le c_2 $ by \eqref{eq:epsfine}, equation \eqref{eq:epsilon le-ri} holds in the interval $[0,\eps]$. If there were $t\in[R/2,R]$ such that $m(t)=0$ then, by monotonicity of $m(t)$, we would conclude $m(R/2)=0$ as well. So we assume $m(t)>0$ in $[R/2,R]$. Then by
\eqref{eq:f(m(t))le am'(t)} and \eqref{eq:epsilon le-ri}, we get
\[
b/2a\le \frac{m'(t)}{m(t)^{n/{n+1}}},\,\qquad \h^1\text{-a.e.}
\]
Integrating between $R/2$ and $R$ we get by \eqref{eq:m(t)<eps}
\[
bR/4a\le(m(R)^{1/{(n+1)}}-m(R/2)^{1/{(n+1)}})\le m(R)^{1/{(n+1)}}\le (\eps\ell_2)^{1/(n+1)}R.
\]
This is a contradiction, since $\eps\ell_2<(b/4a)^{n+1}=I_C(v)^{n+1}/(8^{n+1} v^n)\le \ell_2^{n+1}/8^{n+1}$ by \eqref{eq:epsfine} and Proposition \ref{prp:ICleICmin}. So the proof in case $h(x,R)=\vol{E\cap B_C(x,R)}\,(\vol{B_C(x,R))}^{-1}$ is completed.
For the remaining case, when $h(x,R)=\vol{B_C(x,R)}^{-1}\vol{B_C(x,R)\setminus E}$, we use Lemma \ref{lem:I_C(v)> cv}  and the fact that $I_C$ is non-decreasing proven in Proposition \ref{prp:excylindical}. Then we argue as in Case~1 in Lemma~4.2 of \cite{le-ri} to get
\[
c_1/4\le (\eps\ell_2)^{1/(n+1)}.
\]
This is a contradiction, since $\eps\ell_2<(c_1/4)^{n+1}$ by assumption \eqref{eq:epsfine}
\end{proof}

\begin{proposition}
\label{prp:lodenbn}
Let $C\subset \rr^{n+1}$ be a cylindrically bounded convex body and $C_\infty$ its asymptotic cylinder. Assume that both $C$ and $C_\infty$ have smooth boundary. Then there exists a constant $c>0$ such that, for each isoperimetric region $E$ of volume $v>1$,
\begin{equation}
\label{eq:lodenbn}
\pp(E,B_C(x,r))\ge cr^n,
\end{equation}
for $r\le 1$ and $x\in\ptl_C E$.
\end{proposition}

\begin{proof}
Let $E\subset C$ be an isoperimetric region of volume larger than $1$. Choose $\eps>0$ satisfying \eqref{eq:epsfine}. Since $x\in\ptl_C E$ we have $\lim_{r\to 0} h(x,r)\neq 0$ and, by Theorem~\ref{thm:leon rigot lem 42cyl}, $h(x,r)\ge\eps$ for $0<r\le 1$. So we get
\begin{equation*}
\begin{split}
\pp(E,B_C(x,r))&\ge M\min\{\vol{E\cap B_C(x,r)}, \vol{B_C(x,r) \setminus E}\}^{n/{(n+1)}}
\\
&=M\,(\vol{B_C(x,r)}\,h(x,r))^{n/(n+1)}\ge M(\vol{B_C(x,r)}\,\eps)^{n/(n+1)}
\\
&\ge M\,(\ell_1\eps)^{n/(n+1)}\,r^n.
\end{split}
\end{equation*}
Inequality \eqref{eq:lodenbn} follows by taking $c=M(\ell_1\eps)^{n/(n+1)}$, which is independent of $v$.
\end{proof}
\begin{remark}
\label{rem:isocyl}
Theorem \ref{thm:leon rigot lem 42cyl} and Proposition \ref{prp:lodenbn} also hold if $C$ is a convex cylinder.
\end{remark}
As a Corollary we obtain a new proof of Theorem \ref{thm:isocyl}
\begin{corollary}
\label{cor:isocyl}
Let $C=K\times\rr$, where $K\subset\rr^n$ is a convex body. Then there is a constant $v_0>0$ so that $I_C(v)=2\h^n(K)$ for all $v\ge v_0$. Moreover, the slabs $K\times [t_1,t_2]$ are the only isoperimetric regions of volume larger than or equal to $v_0$.
\end{corollary}

\begin{proof}
Let $E$ be an isoperimetric region with volume
\begin{equation}
\label{eq:ekstod2}
\vol{E}>2mr_0 \h^n(K),
\end{equation}
where $r_0$, $c>0$, are the constants in Proposition~\ref{prp:lodenbn} (see also Remark~\ref{rem:isocyl}), and $m>0$ is chosen so that
\begin{equation}
\label{eq:ekstod3}
mcr_0^n>2\h^n(K).
\end{equation}
By results of Talenti on Steiner symmetrization for finite perimeter sets \cite{talenti}, we can assume that the boundary of $E$ is the union of two graphs, symmetric with respect to a horizontal hyperplane, over a subset $K^*\subset K$. If $K^*=K$ then $P_C(E)\ge 2 \h^n(K)$, since the orthogonal projection over $K\times\{0\}$ is perimeter non-increasing. This implies $P_C(E)= 2 \h^n(K)$ and it follows, as in the proof of Theorem~\ref{thm:isocyl}, that $E$ is a slab.

So assume that $K^*$ is a proper subset of $K$. Since $\vol{E}>2mr_0 \h^n(K)$, $E$ cannot be contained in the slab $K \times[-r_0m, r_0m]$. Then as $\ptl_C E$ is a union of two graphs over $K^*$  we can find $x_j\in \ptl_{C} E$, $1\le j \le m$, so that the balls centered at these points are disjoint. Then by the lower density bound \eqref{eq:lodenbn} we get
\begin{equation}
\label{eq:ekstod2}
\pp_C(E)\ge \sum_{j=1}^m \pp(E,B_C(x_j,r_0))\ge mcr_0^n>2\h^n(K),
\end{equation}
a contradiction since $I_C\le 2\h^n(K)$. 
\end{proof}

Recall that, in Corollary~\ref{cor:half-cylinder}, we showed that, given a half-cylinder $K\times [0,\infty)$, there exists $v_0>0$ so that every isoperimetric region in $K\times [0,\infty)$ of volume larger than or equal to $v_0$ is a slab $K\times [0,b]$, where $b=v/H^n(K)$. We can use this result to obtain

\begin{theorem}
\label{thm:thmcyli}
Let $C\subset \rr^{n+1}$ be a cylindrically bounded convex body, $C_{\infty}=K\times\rr$ its asymptotic cylinder and $C_\infty^+=K\times [0,\infty)$. Let $\{E_i\}_{i\in\nn}$ be a sequence of isoperimetric regions with $\lim_{i\to\infty}|E_i|=\infty$.

Then truncated downward translations of $E_i$ converge in Hausdorff distance to a half-slab $K\times [0,b]$ in $C_\infty^+$. The same convergence result holds for their free boundaries.
\end{theorem}

\begin{proof}
By Corollary~\ref{cor:half-cylinder}, we can choose $v_0>0$ such that each isoperimetric region with volume $v\ge v_0$ in $C_+^{\infty}$ is a half-slab $K\times [0,b(v)]$ of perimeter $\h^n(K)$, where $b(v)=v/H^n(K)$.

Since $|E_i|\to\infty$, we can find vertical vectors $y_i$, with $|y_i|\to\infty$, so that $\Om_i=(-y_i+E_i)\cap \{x_{n+1}\ge 0\}$ has volume $v_0$ for large enough $i\in\nn$. We observe also that, by Proposition~\ref{prp:lodenbn} and the fact that $I_C\le H^n(K)$, the sets $\ptl E_i$ have uniformly bounded diameter.

Consider the convex bodies
\begin{equation}
\label{eq:tod cyli2}
C_i=(-y_i+ C )\cap \{x_{n+1}\ge 0\},
\end{equation}
for $i\in\nn$. The sets $C_i$ have the same asymptotic cylinder $C_\infty$ and we have
\begin{equation}
\label{eq:cupdotcyli2}
\bigcup_{i\in\nn} C_i=C^+_{\infty}.
\end{equation}

By construction we have
\begin{equation}
\label{eq:cup tod cyli 6}
P_{C_i}(\Om_i)\le P_C(E_i)\le H^n(K).
\end{equation}
Since $\ptl E_i$ are uniformly bounded and $\vol{\Om_i}=v_0$, there exists a Euclidean geodesic ball $B$ such that $\Om_i\subset B$ for all $i\in\nn$. By \eqref{eq:cupdotcyli2} the sequence of convex bodies $\{C_i\cap B\}_{i\in\nn}$ converges to $C_\infty^+\cap B$ in Hausdorff distance and, by \cite[Theorem 3.4]{rv1}, in lipschitz distance.
Hence, by the proof of \cite[Theorem 3.4]{rv1} and \cite[Lemma 2.3]{rv1}, we conclude  there exists a finite perimeter set $\Om\subset C_\infty^+$, such that
\begin{equation}
\label{eq:cup tod cyli 7}
\Om_i\stackrel{L^1}{\to}\Om\quad\text{and}\quad P_{C_\infty^+}(\Om)\le\liminf_{i\to\infty} P_{C_i}(\Om_i).
\end{equation}
So we obtain from \eqref{eq:cup tod cyli 6} and \eqref{eq:cup tod cyli 7},
\begin{equation}
H^n(K)=I_{C_\infty^+}(v_0)\le P_{C_\infty^+}(\Om)\le\liminf_{i\to\infty} P_{C_i}(\Om_i)
\le\liminf_{i\to\infty} P_C(E_i)\le H^n(K),
\end{equation}
what implies that $\Om$ is an isoperimetric region of volume $v_0$ in $C_\infty^+$ and so it is a slab.


Furthermore, the arguments of \cite[Theorem 5.11]{rv1} and \cite[Theorem 5.13]{rv1} can be applied here to improve the $L^1$ convergence to Hausdorff convergence, both for the sets $\Om_i$ and for their free boundaries.
\end{proof}

\begin{remark}
The proof of Theorem~\ref{thm:thmcyli} implies $\displaystyle\lim_{v\to\infty}I_C(v)=\h^n(K)$. So we have a different proof of Theorem~\ref{thm:asym-profile}.
\end{remark}

\bibliography{convex}
\end{document}